\documentclass[12pt, a4paper, reqno]{amsart}
\usepackage[centering, heightrounded]{geometry}
\usepackage{amsfonts,amssymb,amsmath,amsgen,amsthm,color}

\usepackage{hyperref}

\theoremstyle{plain} \newtheorem{theorem}{Theorem} [section]
\newtheorem{definition}[theorem]{Definition}

\newtheorem{lemma}[theorem]{Lemma}

 \theoremstyle{remark}
\newtheorem{remark}[theorem]{Remark}

\newtheorem*{remark*}{Remark}

\def\C{{\mathbb C}}
\def\R{{\mathbb R}}
\def\Z{{\mathbb Z}}
\def\T{{\mathbb T}}

\def\({\left(} \def\){\right)} \def\<{\left\langle}
\def\>{\right\rangle}  \def\ge{\geqslant}

\def\d{{\partial}} \def\eps{\varepsilon}

\DeclareMathOperator{\RE}{Re} \DeclareMathOperator{\IM}{Im}

\numberwithin{equation}{section}

\renewcommand{\r}{\right}
\newcommand{\pt}{\partial}
\newcommand{\cleq}{\lesssim}

\renewcommand{\Im}{\operatorname{Im}}

\def\tbra[#1,#2]{\left\langle #1 , #2\right\rangle} 
\def\rbra[#1,#2]{\left( #1 , #2 \right)} 

\newcommand{\ce}{\mathrel{\mathop:}=}

\def\norm[#1]{\left\Vert #1 \right\Vert}
\def\abs[#1]{\left\vert #1 \right\vert}

\begin{document}

\title[]{Low regularity solutions to the logarithmic Schr\"odinger equation}
\author[R. Carles]{R\'emi Carles} \address{Univ Rennes, CNRS\\ IRMAR - UMR
  6625\\ F-35000 Rennes, France}
\email{Remi.Carles@math.cnrs.fr}

\author[M. Hayashi]{Masayuki Hayashi}
\address{Dipartimento di Matematica, Universit\`a di Pisa, Largo Bruno Pontecorvo, 5 56127 Pisa, Italy
\newline\indent
Waseda Research Institute for Science and Engineering, Waseda University, Tokyo 169-8555, Japan
}
 \email{masayuki.hayashi@dm.unipi.it}

\author[T. Ozawa]{Tohru Ozawa}
\address{Department of Applied Physics\\ Waseda University\\ Tokyo
  169-8555\\  Japan}
\email{txozawa@waseda.jp}

\begin{abstract}
 We consider the logarithmic Schr\"odinger equation, in various
 geometric settings. We show that the flow map can be uniquely
 extended from $H^1$ to $L^2$, and that this extension is Lipschitz
 continuous. Moreover, we prove the regularity of the flow map in intermediate Sobolev spaces.
\end{abstract}

\maketitle



\section{Introduction}
\label{sec:1}

We consider the Cauchy problem associated to the logarithmic
Schr\"odinger equation
\begin{equation}
\label{eq:1.1}
 i\d_t u +\Delta u +\lambda u\ln\(|u|^2\)=0,
\quad u_{\mid t=0} =\varphi \, , 
\end{equation}
with $x\in \Omega$, $\lambda\in \R$.  The standard occurrence for
$\Omega$ is $\R^d$. In view of the framework considered in numerical
simulations (see e.g. \cite{BCST}), the domain $\Omega$ that we consider
may be:
\begin{itemize}
\item The whole space $\Omega=\R^d$,
\item A half space $\Omega=\R^d_+$, with  zero Dirichlet boundary
  condition, $u_{\mid \d \Omega}=0$, 
\item A general domain $\Omega\subset\R^d$;
    \eqref{eq:1.1} is then considered with  
  zero Dirichlet boundary condition, and the
  Laplacian $\Delta$ is understood to be the self-adjoint realization
  in $H^{-1}(\Omega)$ with the domain $D(\Delta)=H^1_0(\Omega)$ (see
  e.g. \cite[Chapter~2]{CazHar}), 
\item The torus $\Omega=\T^d=\R^d/\Z^d$.
\end{itemize}
In all cases, there is no restriction on the space dimension $d\ge
1$.  Formally,
the mass and energy are independent of time:
\begin{align*} 
& M(u(t))=\|u(t)\|_{L^2(\Omega)}^2,\\ 
&
E(u(t))=\|\nabla u(t)\|_{L^2(\Omega)}^2-\lambda\int_{\Omega}
|u(t,x)|^2\(\ln(|u(t,x)|^2)-1\)dx.
\end{align*}
The equation \eqref{eq:1.1} was introduced in \cite{BiMy76} for
quantum mechanics, and
has attracted the interest of physicists from various fields ever
since (see
e.g. \cite{buljan,DFGL03,hansson,HeRe80,Hef85,KEB00,yasue,Zlo10}). From
the mathematical point of view, the 
first study goes back to \cite{CaHa80}, where the Cauchy problem was
investigated in the case $\Omega=\R^d$. The main results in \cite{CaHa80} we mention here are:
\begin{itemize}
\item Theorem~1.2 b): if $\varphi\in L^2(\R^d)$, then \eqref{eq:1.1}
  has a unique weak solution $u\in C(\R,L^2(\R^d))$ in the sense of Brezis \cite{Brezis}. 
 \item Theorem~2.1: if $\lambda>0$, and $\varphi\in H^1(\R^d)$ is such
   that $|\varphi|^2\ln(|\varphi|^2) \in L^1(\R^d)$, then \eqref{eq:1.1} has a
   unique solution $u\in C(\R,H^1(\R^d))$. In addition,
   $|u(t)|^2\ln(|u(t)|^2)\in L^1(\R^d)$ for all $t\in \R$, and the mass
   and energy of $u$ are independent of time. 
 \end{itemize}
 The goal of this paper is mostly to revisit the first statement
 above. The weak solution in Theorem~1.2 b) is obtained as a limit of the sequence of strong solutions, and the existence of this limit is guaranteed through the maximal monotone theory. For the convenience of the reader, more details are provided in Appendix \ref{sec:A}.

The mathematical study of \eqref{eq:1.1} has been considered since \cite{CaHa80},
regarding both the Cauchy problem 
(\cite{AMS14,CaGa18,HayashiM2018}) and the 
dynamical properties of the solutions
(e.g. \cite{Caz83,Ar16,CaGa18,FEDCDS,FeAIHP}).  More recently, the
Cauchy problem was revisited by the second and third authors in
\cite{HO}, where strong global solutions are
provided in a constructive way, in each of the following functional settings:
\begin{itemize}
\setlength{\itemsep}{3pt}
\item $\varphi\in H^1(\R^d)$ for $\lambda\neq 0$,
\item Energy space: $\varphi \in W_1:=\{f\in H^1(\R^d),\
  |f|^2\ln(|f|^2)\in L^1(\R^d)\}$ for $\lambda\neq 0$, 
\item $H^2$ energy space: $\varphi \in W_2:=\{f\in H^2(\R^d),\
  f\ln(|f|^2)\in L^2(\R^d)\}$ for $\lambda>0$.
\end{itemize}
We emphasize that unlike in the case of more standard power-like
nonlinearities for Schr\"odinger equations, due to the singularity of
the logarithm at the origin, several questions
regarding the Cauchy problem \eqref{eq:1.1} are unclear, such as:
\begin{itemize}
\setlength{\itemsep}{3pt}
\item The propagation of higher regularity: if $\varphi\in H^3$, can we
  guarantee that the solution $u$ remains in $H^3$, even locally in
  time?
 \item Is there a minimal regularity for a ``reasonable''
   notion of solution?
\end{itemize}

In this paper, we focus on the second question. 
When $\Omega=\R^d$, \eqref{eq:1.1} is invariant under Galilean transformations  
\[
u(t,x) \mapsto e^{iv\cdot x -i|v|^2 t}u(t,x-2vt) \quad\text{for}~v \in \R^d,
\]
which leave the $L^2(\R^d)$ norm invariant, so it may be expected that like
for nonlinear Schr\"odinger equations with pure power nonlinearities, the flow map fails to
be uniformly continuous on $H^s(\R^d)$ when $s<0$ (\cite{KPV01,CCT, CCTpre}). See also \cite{CCTpre, CDS12, Ki18} for stronger ill-posedness results. We note, however, that in the case
of the KdV equation the flow map was extended continuously to the level of $H^{-1}(\R)$ (\cite{KV19}), even though it has been known that the flow map cannot be uniformly continuous on $H^s(\R)$ when $s<-3/4$. Similar well-posedness pictures can also be seen in the cubic NLS equation and the modified KdV equation (\cite{HKVpre}). 

Note that there is no
natural scaling associated to \eqref{eq:1.1}, when $\Omega=\R^d$,
of the form $u_\kappa(t,x) = \kappa^\alpha
u(\kappa^\beta t,\kappa^\gamma x)$, which leaves the equation
invariant. Yet, another invariance, rather unique for nonlinear
Schr\"odinger equations, shows that the size of the initial data does
not affect the Cauchy problem, nor  the dynamical properties of the
solution: if $u$ solves \eqref{eq:1.1}, then for any $z\in \C$,
\begin{equation*}
  u_z(t,x) = z u(t,x) e^{i\lambda t\ln|z|^2}
\end{equation*}
solves the same equation, with initial datum $z\varphi$. This unusual
invariance also shows that regardless of the function spaces
considered, the flow map cannot be $C^1$ at the origin, due to the
above oscillating factor: it is at most Lipschitz continuous. 

The main result of this paper is that the flow map defined on $H^1(\Omega)$ can be uniquely extended on $L^2(\Omega)$, this extension is Lipschitz continuous, and preserves possible intermediate Sobolev regularity:
  \begin{theorem}
\label{thm:1.1}
The equation \eqref{eq:1.1} is globally well-posed in $L^2(\Omega)$ in the following sense: The $H^1$ solution map $\Phi$  is uniquely extended to $L^2(\Omega)$, and for $\varphi\in L^2(\Omega)$, $u=\Phi(\varphi)\in C(\R,L^2(\Omega))$ satisfies 
\begin{align*}
i\pt_t u+\Delta
  u+\lambda u\ln(|u|^2)=0\quad\text{in}~H^{-2}(\omega)
\end{align*}
for any open sets $\omega\Subset\Omega$ and all $t\in\R$, with $u_{\mid t=0}=\varphi$.
Moreover, $\Phi$ is Lipschitz continuous: 
\begin{align*}
\norm[ \Phi(\varphi)(t)-\Phi(\psi)(t) ]_{L^2}\leq e^{2|\lambda t|}\norm[\varphi-\psi]_{L^2}
\end{align*}
for any $\varphi,\psi\in L^2(\Omega)$ and all
$t\in\R$. In the case where $\Omega\in \{\R^d,\R^d_+,\T^d\}$, if in addition
$\varphi\in H^s(\Omega)$ for some $s\in (0,1)$, then $\Phi(\varphi)\in
C(\R,H^s(\Omega))$. 
\end{theorem}
Our contributions in this paper can be summarized as follows.
\\[2pt]
\noindent\emph{1.\,Meaning of $L^2$ solutions revisited.}
We show that logarithmic nonlinearities make sense for general functions belonging to $L^\infty((-T,T), L^2(\Omega))$ for $T>0$, which in particular gives the meaning of solutions to \eqref{eq:1.1} in the distribution sense (see Lemma \ref{lem:3.1} and Remark \ref{rem:3.2}). 
This differs from \cite{CaHa80} in that it gives the meaning to $L^2$ solutions regardless of how the solution is constructed or independently of limiting procedures.
\\[2pt]
\noindent\emph{2.\,GWP in $L^2$ independently of the maximal monotone theory.} 
We construct $L^2$ solutions as an extension of the solution map on $H^1$ while preserving $L^2$ Lipschitz flow in \cite{CaHa80} (see Lemma \ref{lem:3.3} below). The formulation of the global well-posedness is inspired from recent results on completely integrable systems (\cite{KV19, HKVpre, HKNVpre}). Note that the $L^2$ Lipschitz flow is a natural consequence of the remarkable inequality
\begin{align}
\label{eq:1.2}
\bigl| \Im\bigl[ (\overline{z_1-z_2})( z_1\log (|z_1|)-z_2\log(|z_2|) ) \bigr]\bigr|
\leq |z_1-z_2|^2\quad\text{for all}~z_1,z_2\in\C,
\end{align}
which was discovered in \cite[Lemme 1.1.1]{CaHa80}.
It may be common to \cite[Theorem~1.2~b)]{CaHa80} that $L^2$ solutions are constructed as a limit of sequences of strong solutions, but
our construction would be regarded to be a more direct consequence of \eqref{eq:1.2} and it is independent of the maximal monotone theory.
\\[2pt]
\noindent\emph{3.\,Intermediate Sobolev regularity.} 
This regularity result is obtained as a new application of the inequality \eqref{eq:1.2}. By using the Sobolev norm based on the difference quotient, we can effectively utilize \eqref{eq:1.2} to obtain an a priori estimates on $H^s$ for all $s\in(0,1)$. The domain restriction here comes from the use of space translation invariance in our proof.
\\[2pt]
\indent
The rest of the paper is organized as follows. In
Section~\ref{sec:2}, we recall the construction of the
$H^1$ solution map, from \cite{HO}. In Section~\ref{sec:3}, we show
that this map can be uniquely extended to $L^2$, as a Lipschitz
continuous map. In Section~\ref{sec:4}, we show that the intermediate
Sobolev regularity is propagated by this map.

\subsection*{Notation}
We sometimes use the abbreviated notation such as
\begin{align*}
C_T(X)=C([-T,T] ,X),\quad L^\infty_T(X)=L^\infty((-T,T) ,X)
\end{align*}
for $T>0$ and a Banach/Fr\'echet space $X$. For open sets $\omega, \Omega\subset\R^d$, we write $\omega\Subset\Omega$ if $\bar{\omega}\subset\Omega$ and $\bar{\omega}$ is compact, where $\bar{\omega}$ is the closure of $\omega$ in $\R^d$.

According to  \cite{DPV}, we define the fractional Sobolev spaces $H^s$. 
For a general open set $\Omega\subset\R^d$, the fractional Sobolev spaces $H^s(\Omega)$ for $s\in(0,1)$ are defined via the norm
\begin{align}
\label{eq:1.3}
\norm[f]_{H^s(\Omega)}^2=\norm[f]_{L^2(\Omega)}^2+ \iint_{\Omega\times\Omega}
\frac{|f(x)-f(y)|^2}{|x-y|^{d+2s}}dxdy.  
\end{align}
We define $H^s_0(\Omega)$ by the closure of $C^\infty_c(\Omega)$ in the norm $\norm[\cdot]_{H^s(\Omega)}$ and $H^{-s}(\Omega)$ by the dual space of $H^s_0(\Omega)$ for $s\in(0,1)$. When $\Omega=\R^d$, the Sobolev space via the norm \eqref{eq:1.3} coincides with Bessel potential spaces endowed with the norm
\begin{align}
\label{eq:1.4}
\norm[f]_{H^s(\R^d)}^2=\norm[ (1+4\pi^2|\xi|^2)^s \hat{f}(\xi) ]_{L^2(\R^d)}^2
\quad\text{for}~s\in\R,
\end{align}
where $\hat{f}(\xi)$ is the Fourier transform defined by
\begin{align*}
\hat{f}(\xi) =\int_{\R^d} f(x)e^{-2\pi ix\cdot\xi}dx\quad\text{for}~\xi\in\R^d.
\end{align*}
Similarly to Bessel potential spaces on $\R^d$, the Sobolev spaces $H^s(\T^d)$ on the torus are defined via the norm
\begin{align}
\label{eq:1.5}
\|f\|_{H^s(\T^d)}^2 = \sum_{n\in \Z^d}\(1+4\pi^2|n|^2\)^s |\hat f(n)|^2\quad\text{for}~s\in\R,
\end{align}
where $\hat{f}(n)$ is the discrete Fourier transform defined by
\begin{align*}
\hat{f}(n) =\int_{\T^d} f(x)e^{-2\pi ix\cdot n}dx \quad\text{for}~n\in\Z^d.
\end{align*}

We use $A\cleq B$ to denote the inequality $A\leq CB$ for some constant $C>0$. The dependence of $C$ 
is usually clear from the context and we often omit this
dependence. We may sometimes write $A\cleq_*B$ to clarify the dependence of the implicit constant.


%


\section{Global $H^1$ solutions}
\label{sec:2}

In this section we review global $H^1$ solutions to \eqref{eq:1.1}. 
Here let $\Omega\subset\R^d$ be a general domain and let $\varphi\in H^1_0(\Omega)$. 
Following the strategy of \cite{HO}, we
regularize \eqref{eq:1.1} like in \cite{BCST}: for $\eps >0$, we consider
approximate solutions $u^\eps$ to  
\begin{equation}
  \label{eq:2.1}
  i\d_t u^\eps+\Delta u^\eps+2\lambda u^\eps \ln\(
  |u^\eps|+\eps\)= 0,\quad u^\eps_{\mid t=0}=\varphi.
\end{equation}
As the nonlinearity is now smooth, with moderate growth at infinity,
the equation \eqref{eq:2.1} has a unique, global solution (see
e.g. \cite[Chapter~3]{CazCourant}, \cite[Section 2.1]{HO}),
\begin{equation*}
  u^\eps\in C(\R,H_0^1(\Omega))\cap C^1(\R,H^{-1}(\Omega)).
\end{equation*}
Uniqueness and regularity in time in the case of
\eqref{eq:2.1} follow from the generalization of the inequality \eqref{eq:1.2}, generalized successively in
\cite{BCST} and \cite{HO}. We state the most general version: 
\begin{lemma}[{\cite[Lemma~A.1]{HO}}]\label{lem:2.1}
  For all $z_1,z_2\in \C$,
  $\eps_1,\eps_2\ge 0$, 
\begin{align*} 
\left| \IM\bigl[ \(\overline{z_1-z_2}\)\(z_1\ln\(|z_1|+\eps_1\) -
z_2\ln\(|z_2|+\eps_2\)\) \bigr] \r|&\leq |z_1-z_2|^2\\
&\quad+|\eps_1-\eps_2|\times |z_1-z_2|.  
\end{align*}
\end{lemma}
The uniform energy estimate
\begin{equation*}
  \|\nabla u^\eps(t)\|^2_{L^2}\leq e^{4\abs[\lambda t] }\|\nabla \varphi\|^2_{L^2},
\end{equation*}
is easily obtained by differentiating the left hand side in time and invoking Gronwall's lemma. Combined this with Lemma \ref{lem:2.1}, one can prove that $\{u^\eps\}$ forms a Cauchy sequence in $C_T(L^2_{\rm loc}(\Omega))$ as $\eps\downarrow0$ for any $T>0$. Thus, we obtain the following result.
\begin{theorem}[From Theorem~4.1 in \cite{HO}]\label{thm:2.2}
  For any $\varphi\in H^1_0(\Omega)$, there exists a unique solution $u\in
  C(\R,H^1_0(\Omega))$ to \eqref{eq:1.1}, in the sense of
  \begin{equation*}
    i\d_t u +\Delta u+\lambda u\ln\(|u|^2\)=0\quad\text{in
    }H^{-1}(\omega) ,
  \end{equation*}
  for all $\omega\Subset \Omega$ and all $t\in \R$, and with
  $u_{\mid t=0}=\varphi$. 
\end{theorem}
The whole argument holds true for the torus $\T^d$ (not considered in
\cite{HO}), in the same way, and we obtain:
\begin{theorem}
\label{thm:2.3}
  For any $\varphi\in H^1(\T^d)$, there exists a unique solution $u\in
  C(\R,H^1(\T^d))$ to \eqref{eq:1.1}, in the sense of
  \begin{equation*}
    i\d_t u +\Delta u+\lambda u\ln\(|u|^2\)=0\quad\text{in
    }H^{-1}(\T^d) ,
  \end{equation*}
  for all $t\in \R$, and with
  $u_{\mid t=0}=\varphi$. 
\end{theorem}


\section{The Cauchy problem in $L^2$} 
\label{sec:3}

In this section we construct strong $L^2$ solutions to \eqref{eq:1.1}. 
For convenience of notation, here we only consider the case where $\Omega$ is a general domain in $\R^d$. 
In the case of the torus $\T^d$, the same argument still works if we replace $H^1_0(\Omega)$ by $H^1(\T^d)$.

The first task in order to consider \eqref{eq:1.1} with initial data
$\varphi\in L^2(\Omega)$ is to clarify in what sense the $L^2$ solution satisfies the equation. 
\begin{lemma}
\label{lem:3.1}
Let $u\in L^\infty_T(L^2(\Omega))$ for $T>0$. Then, for any small $\eps>0$ and all $\omega\Subset\Omega$, the nonlinear term satisfies $u\ln(|u|^2)\in L^\infty_T(H^{-\eps}(\omega))$.
\end{lemma}
\begin{proof}
We note that 
\begin{align*}
\abs[u\ln(|u|^2)]\cleq |u|^{1-\delta}+|u|^{1+\delta},
\end{align*}
for any $\delta\in (0,1)$. Writing, for $\psi$ a test function
supported in $\omega\Subset \Omega$,
\begin{align*} \left|\<u\ln |u|^2,\psi\>\right|&
\lesssim
\int_{\omega} |u|^{1-\delta}|\psi| + \int_{\omega} |u|^{1+\delta}|\psi|\\
&\lesssim \|u\|_{L^{p'(1-\delta)}}^{1-\delta}\|\psi\|_{L^{p}} 
+ \|u\|_{L^{q'(1+\delta)}}^{1+\delta} \|\psi\|_{L^q},
\end{align*}
and considering $p'(1-\delta)=2$, that is $p=\frac{2}{1+\delta}$, and
$q=\frac{2}{1-\delta}$, we get
\begin{align*}
 \left|\<u\ln(|u|^2),\psi\>\right|&\lesssim \|u\|_{L^2}^{1-\delta}
 \|\psi\|_{L^{\frac{2}{1+\delta}}(\omega)} +  \|u\|_{L^2}^{1+\delta}
  \|\psi\|_{L^{\frac{2}{1-\delta}}(\omega)} \\
  &\lesssim \|u\|_{L^2}^{1-\delta} \|\psi\|_{L^2(\omega)} +
    \|u\|_{L^2}^{1+\delta}  \|\psi\|_{H^{d\delta/2}(\omega)},
\end{align*}
where we have used the Sobolev embedding in the last inequality.\footnote{More precisely, we first consider the zero extension of $\psi$, and then apply the Sobolev embedding $H^{d\delta/2}(\R^d)\subset L^{\frac{2}{1-\delta}}(\R^d)$.} Therefore, we deduce that $u\ln(|u|^2)\in H^{-d\delta/2}(\omega)$, which proves the result.
\end{proof}
\begin{remark}
\label{rem:3.2}
If $u\in L^{\infty}_T (L^2(\Omega))$, then the
equation \eqref{eq:1.1} makes sense in the sense of 
\begin{align}
\label{eq:3.1}
i\d_t u+\Delta u+\lambda u \ln(|u|^2)=0\quad\text{in}~H^{-2}(\omega),
\end{align}
for any $\omega\Subset\Omega$ and a.e. $t$. In particular, 
it gives the meaning of the equation in the distribution sense for any $u\in L^\infty_T(L^2(\Omega))$.
\end{remark}
We now recall the following important lemma:
\begin{lemma}[{\cite[Lemma 2.2.1]{CaHa80}}]
\label{lem:3.3}
Assume that $u, v\in C_T(H^1_0(\Omega))$ satisfies \eqref{eq:1.1}
in the distribution sense. Then, we have
\begin{align}
\label{eq:3.2}
\norm[u(t)-v(t)]_{L^2}\leq e^{2|\lambda t|}\norm[u(0)-v(0)]_{L^2}
\quad\text{for}~t\in[-T,T].
\end{align}
\end{lemma}
\begin{proof}
For completeness, we give a proof of this result. To simplify the presentation, we
consider only the case $\Omega=\R^d$. We set
\begin{align*}
M\ce \max\left\{ \norm[u]_{C_T(H^1)},  \norm[v]_{C_T(H^1)}\r\}.
\end{align*}
We note that $u,v$ satisfy the equation in the sense of 
\begin{align}
\label{eq:3.3}
i\pt_t u+\Delta u+\lambda u\ln(|u|^2)=0\quad\text{in}~H^{-1}(B_R)
\end{align}
for any $R>0$ and for all $t\in\R$, where $B_R$ is the open ball of radius $R$ with center at the origin of $\R^d$.
Take a function $\zeta\in C^{\infty}_c(\R^d)$ satisfying
\begin{align*}
\zeta (x)=
\left\{
\begin{aligned}
&1&&\text{if}~|x|\leq 1,
\\
&0 &&\text{if}~|x|\geq 2,
\end{aligned}
\r.
\qquad 0\leq\zeta (x)\leq 1\quad\text{for all}~x\in\R^d.
\end{align*}
We set $\zeta_R(\cdot)=\zeta(\cdot/R)$ for $R>0$.
It follows from \eqref{eq:3.3} and Lemma~\ref{lem:2.1} (with
$\eps_1=\eps_2=0$) that 
\begin{align*}
\frac{d}{dt}\norm[\zeta_R(u-v)]_{L^2}^2
&=2\Im\tbra[i\zeta_R^2\pt_t(u-v),u-v]_{H^{-1}(B_{2R}), H^1_0(B_{2R})}
\\
&=
\begin{aligned}[t]
&2\Im\rbra[\nabla(\zeta_R^2)\nabla(u-v),u-v]_{L^2}
\\
&{}-4\lambda\Im\rbra[\zeta_R^2
\left(u\ln |u|-v\ln|v|\r),
u-v]_{L^2}
\end{aligned}
\\
&\leq 
\frac{C(M)}{R}+4\abs[\lambda]\norm[\zeta_R(u-v)]_{L^2}^2.
\end{align*}
Integrating the last inequality over $[0,t]$, and applying Gronwall's lemma, 
\begin{align*}
\norm[\zeta_R(u-v)(t)]_{L^2}^2\leq e^{4|\lambda t|} \(\norm[u(0)-v(0)]_{L^2}^2+\frac{C(M)}{R}T \),
\end{align*}
for all  $t\in (-T,T)$. Applying Fatou's lemma, 
\begin{align*}
\norm[u(t)-v(t)]_{L^2}^2
\leq\liminf_{R\to\infty}\norm[\zeta_R(u-v)(t)]_{L^2}^2\leq
e^{4|\lambda t|} \norm[u(0)-v(0)]_{L^2}^2,
\end{align*}
which proves \eqref{eq:3.2}.
\end{proof}
\begin{remark}
\label{rem:3.4}
 For the proof of Lemma \ref{lem:3.3}, we need to assume
 $H^1$ solutions to give a sense of the duality product. Note,
 however, that \eqref{eq:3.2} remains meaningful for
 $L^2$ solutions. 
\end{remark}
 
Now we take $\varphi\in L^2(\Omega)$ as the initial data.
Take $\{\varphi_n\}\subset H^1_0(\Omega)$ such that $\varphi_n\to
\varphi~\text{in}~L^2(\Omega)$. We know from Theorem~\ref{thm:2.2}
that there exists a unique 
solution $u_n\in C(\R, H^1_0(\Omega))$ of \eqref{eq:1.1} with
$u_n(0)=\varphi_n$. Then, it follows from
\eqref{eq:1.1} that $\{u_n\}$ forms a Cauchy sequence in 
$L^\infty_{\rm loc}(\R, L^2(\Omega))$. Therefore we deduce that there exists $u\in
C(\R,L^2(\Omega))$ such that 
\begin{align*}
u_n \to u\quad \text{in}~L^\infty_{\rm loc}(\R, L^2(\Omega)).
\end{align*}

The rest of the proof consists in verifying that $u$ is an
$L^2$ solution in the above sense. 
\begin{lemma}
\label{lem:3.5}
For any $\omega\Subset\Omega$ and for any $\eps>0$, we have
\begin{align*}
g(u_n) \to g(u)\quad\text{in}~L^\infty_{\rm loc}(\R, H^{-\eps}(\omega)),
\end{align*}
where we have set $g(z)=z\ln\(|z|^2\)$. 
\end{lemma}
\begin{proof}
We take a function $\theta\in C^{1}_c(\C,\R)$ satisfying
\begin{align*}
\theta (z)=
\left\{
\begin{aligned}
&1&&\text{if}~|z|\leq 1,
\\
&0 &&\text{if}~|z|\geq 2,
\end{aligned}
\r.
\qquad 0\leq\theta (z)\leq 1\quad\text{for}~z\in\C.
\end{align*}
We set 
\begin{align*}
g_1(u)=\theta(u)g(u), \quad g_2(u)=(1-\theta(u))g(u).
\end{align*}
We note that for $\alpha\in(0,1)$,
\begin{align}
\label{eq:3.4}
\abs[g_1(z)-g_1(w)]&\cleq_\alpha |z-w|^\alpha,
\\
\label{eq:3.5}
\abs[g_2(z)-g_2(w)]&\cleq \(\ln^+|z|+\ln^+|w|\r)|z-w|,
\end{align}
for any $z,w\in\C$. It follows from \eqref{eq:3.4} that
\begin{align}
\label{eq:3.6}
g_1(u_n)\to g_1(u)\quad \text{in}~L^\infty_{\rm loc}(\R, L^2(\omega)).
\end{align}

Regarding the convergence of $g_2(u_n)$, we use the argument in the proof of Lemma \ref{lem:3.1}.
 For $\psi\in C^1_c(\Omega)$ and any $\delta>0$, we obtain
from \eqref{eq:3.5} that 
\begin{align*}
\abs[\int (g_2(u_n)-g_2(u))\psi]&\cleq\int (|u_n|^\delta+|u|^\delta)|u_n-u|\abs[\psi]
\\
&\cleq\(\norm[u_n]_{L^2}^\delta+ \norm[u]_{L^2}^\delta\r)\norm[u_n-u]_{L^2}\norm[\psi]_{L^{\frac{2}{1-\delta}} },
\end{align*}
where we have used H\"older inequality with the exponent relation
\begin{align*}
\frac{\delta}{2}+\frac{1}{2}+\frac{1-\delta}{2}=1.
\end{align*}
Thus, we obtain
\begin{align*}
\abs[\int (g_2(u_n)-g_2(u))\psi]&\leq C(\norm[\varphi]_{L^2})\norm[u_n-u]_{L^2}\norm[\psi]_{H^\eps},
\end{align*}
where we take $\eps=d\delta/2$ by the Sobolev embedding. Therefore, we deduce that
\begin{align*}
\norm[g_2(u_n)-g_2(u)]_{H^{-\eps}}\cleq \norm[u_n-u]_{L^2},
\end{align*}
which implies that
\begin{align}
\label{eq:3.7}
g_2(u_n)\to g_2(u)\quad \text{in}~L^\infty_{\rm loc}(\R, H^{-\eps}(\Omega)).
\end{align}
Hence, the result follows from \eqref{eq:3.6} and \eqref{eq:3.7}.
\end{proof}
We recall that $u_n$ satisfies
\begin{equation}\label{eq:3.8}
  i\d_t u_n +\Delta u_n+\lambda u_n\ln\(|u_n|^2\)=0  \quad
  \text{in }H^{-1}(\omega),
\end{equation}
for all $\omega\Subset \Omega$. 
We now fix $\omega\Subset \Omega$, and take $\psi\in C^{\infty}_c(\omega)$ and $\phi\in C^1_c(\R)$. It follows from \eqref{eq:3.8} that
\begin{align*}
\int_\R \rbra[i u_n,\psi]_{L^2}\phi'(t)dt
&=-\int_\R \tbra[i\pt_t u_n,\psi]_{H^{-1},H^1_0}\phi(t)dt
\\
&=\int_\R\( \rbra[u_n,\Delta\psi]_{L^2}+\tbra[\lambda
     g(u_n),\psi]_{H^{-1}, H^1_0} \) \phi(t)dt.
\end{align*}
Passing to the limit as $n\to\infty$, we obtain from Lemma \ref{lem:3.5} that
\begin{align*}
\int_\R \rbra[i u,\psi]_{L^2}\phi'(t)dt
&=\int_\R\( \rbra[ u,\Delta\psi]_{L^2}-\tbra[\lambda g(u),\psi]_{H^{-1}, H^1_0} \) \phi(t)dt
\\
&=\int_\R\tbra[ \Delta u-\lambda g(u),\psi]_{H^{-2}, H^2_0}\phi(t)dt.
\end{align*}
It is easily verified from this formula that 
\begin{align*}
u\in C(\R,L^2(\Omega))\cap C^1(\R, H^{-2}(\omega)),
\end{align*}
for any $\omega\Subset\Omega$, and $u$ satisfies \eqref{eq:3.1} for any
$\omega\Subset\Omega$ and all $t\in\R$. 
 
The $L^2$ solution just constructed can be regarded as a unique extension of the solution map in $H^1$. 
We define the solution map from $H^1$ initial data by
\begin{align}
\label{eq:3.9}
\Phi: H^1_0(\Omega)\ni\varphi \mapsto u\in C(\R, H^1_0(\Omega)).
\end{align} 
For any $\varphi\in L^2(\Omega)$, we take a sequence $\{\varphi_n\}\subset H^1_0(\Omega)$ such
that $\varphi_n\to\varphi$ in $L^2(\Omega)$ and define 
\begin{align*}
\Phi(\varphi)=\lim_{n\to\infty}\Phi(\varphi_n) \in C(\R, L^2(\Omega)).
\end{align*}
From the above discussion, $\Phi(\varphi)$ yields an $L^2$ solution of
\eqref{eq:1.1}. We note from \eqref{eq:3.2} that $\Phi(\varphi)$ is
defined independently of the approximate sequence
$\varphi_n\to\varphi$.  
Therefore the solution map \eqref{eq:3.9} is uniquely extended from
$H^1_0(\Omega)$ to $L^2(\Omega)$, and hence the first part of
Theorem~\ref{thm:1.1} follows.


\section{Intermediate Sobolev regularity}
\label{sec:4}

To prove the last part of Theorem~\ref{thm:1.1}, we show that the
flow map associated to \eqref{eq:2.1} propagates
$H^s$ regularity for $s\in (0,1)$, uniformly in $\eps\in (0,1]$. For domains $\Omega\subset\R^d$ the fractional Sobolev spaces
space $H^s(\Omega)$ may be defined either by real/complex interpolation
between $L^2(\Omega)$ and $H^1_0(\Omega)$. When $\Omega=\R^d$, it is well known that the Bessel potential spaces by \eqref{eq:1.4} are characterized by complex interpolation as
\begin{align*}
H^s(\R^d)=[L^2(\R^d), H^1(\R^d)]_{s},\quad s\in(0,1). 
\end{align*}
If $\Omega$ is the whole space, a half
space, or a smooth bounded domain with bounded boundary, \cite[Theorem~7.48]{Adams} states
that the fractional Sobolev spaces defined by real interpolation are
equivalent to the ones equipped with the norm
\eqref{eq:1.3}.
If $\Omega=\T^d$, it is also
known that a similar equivalence relation holds as follows,
as can be proven essentially by replacing Plancherel's identity in the case of $\R^d$ (in $x$, in \eqref{eq:4.1}
  below) 
  with Parseval's identity on $\T^d$. 
\begin{lemma}[{\cite[Proposition 1.3]{BO13}}]
\label{lem:4.1}
Let $s\in(0,1)$. Then, for $f\in H^s(\T^d)$ we have the relation 
  \begin{align*}
  \norm[f]_{H^s(\T^d)}^2 \sim \norm[f]_{L^2(\T^d)}^2+
  \iint_{\T^d\times[-\frac{1}{2},\frac{1}{2})^d} \frac{|f(x+y)-f(x)|^2}{|y|^{d+2s}}dxdy.
  \end{align*}
\end{lemma}
We denote the approximate nonlinearity by
\begin{equation*}
  g_\eps(z)= 2z\ln\(|z|+\eps\) \quad\text{for}~\eps\in(0,1]. 
\end{equation*}
If $\Omega=\R^d$, 
by changing variables in \eqref{eq:1.3}, the $H^s$ norm  can be
rewritten as
  \begin{equation}\label{eq:4.1}
      \norm[f]_{H^s(\R^d)}^2 =\norm[f]_{L^2(\R^d)}^2+
  \iint_{\R^d\times\R^d} \frac{|f(x+y)-f(x)|^2}{|y|^{d+2s}}dxdy,
  \end{equation}
for $f\in H^s(\R^d)$.
Then, in view of the conservation of mass, we obtain 
\begin{align*}
  &\frac{d}{dt}\|u^\eps(t)\|_{ H^s(\R^d)}^2 
  \\
  ={}& 2\RE
  \iint_{\R^d\times\R^d} \overline{
   \( u^\eps(t,x+y)-u^\eps(t,x) \)}\d_t\(  u^\eps(t,x+y)-u^\eps(t,x)\) \frac{dxdy}{|y|^{d+2s}}
   \\ 
  ={}& 
\begin{aligned}[t]
 &-2\IM\iint_{\R^d\times\R^d} \overline{ \(
    u^\eps(t,x+y)-u^\eps(t,x) \)} \( \Delta u^\eps(t,x+y)-\Delta u^\eps(t,x)\)
 \frac{dxdy}{|y|^{d+2s}} \\ 
   &{}-2\lambda \IM\iint_{\R^d\times\R^d}
  \overline{
    \(u^\eps(t,x+y)-u^\eps(t,x) \)}\(g_\eps\( 
    u^\eps(t,x+y)\) - g_\eps\( u^\eps(t,x)\)\) \frac{dxdy}{|y|^{d+2s}}.
\end{aligned}
\end{align*}
The first term on the right hand side is zero by integration by parts
in $x$.
For the second term, by applying Lemma~\ref{lem:2.1} we obtain
\begin{align*}
  &\frac{d}{dt}\|u^\eps(t)\|_{H^s(\R^d)}^2 
  \\
  \leq {}&2|\lambda|
  \iint_{\R^d\times\R^d} \left|\IM \left[ \overline{
  \(u^\eps(t,x+y)-u^\eps(t,x) \)}\(g_\eps\(  u^\eps(t,x+y)\) - g_\eps\(
  u^\eps(t,x)\)\) \right] \right|\frac{dxdy}{|y|^{d+2s}}
  \\
  \leq{}&4|\lambda|\iint_{\R^d\times\R^d} \left|
    u^\eps(t,x+y)-u^\eps(t,x)\right|^2 \frac{dxdy}{|y|^{d+2s}}
    \leq 4\abs[\lambda]\norm[u^\eps(t)]_{H^s(\R^d)}^2. 
\end{align*}
Therefore, by Gronwall's lemma we deduce
\begin{equation}
\label{eq:4.2}
  \|u^\eps(t)\|_{H^s(\R^d)}^2\leq e^{4|\lambda t|}\|\varphi\|_{H^s(\R^d)}^2
  \quad\text{for all}~t\in\R.
\end{equation}

  If $\Omega=\R^d_+$, we extend $u(t,\cdot)$ to $\R^d$ by symmetry,
  according to \cite[Remark~2.7.2]{CazCourant}: introduce $\tilde
  \varphi$, defined on $\R^d$ by
  \begin{equation*}
    \tilde \varphi(x) =
    \left\{
      \begin{aligned}
       &~\varphi(x_1,\dots ,x_d) &&\text{if }x_d>0,\\
      &{-}\varphi(x_1,\dots ,-x_d) &&\text{if }x_d<0.
      \end{aligned}
    \right.  
  \end{equation*}
By uniqueness, $u$ is the restriction to $\R^d_+$ of the solution
$\tilde u$ to \eqref{eq:1.1} on $\R^d$ with initial datum $\tilde
\varphi$,  as we check that $\tilde
  u(t,x_1,\dots,x_d)= -\tilde
  u(t,x_1,\dots,-x_d)$ a.e., so the Dirichlet
  boundary condition is satisfied. Therefore, the desired estimates for the solution $u$ follow from \eqref{eq:4.2} for the extended solution $\tilde u$.

If $\Omega=\T^d$, by Lemma \ref{lem:4.1} one can define the equivalent $H^s$ norm by
\begin{align*}
 \|f\|_{ \tilde{H}^s(\T^d) }^2 =\|f\|_{L^2(\T^d)}^2+  \iint_{\T^d\times[-\frac{1}{2},\frac{1}{2})^d}
\frac{|f(x+y)-f(x)|^2}{|y|^{d+2s}}dxdy.
\end{align*}
Using this norm, as in the case of $\R^d$ we obtain
\begin{align}
\label{eq:4.3}
 \|u^\eps(t)\|_{\tilde{H}^s(\T^d)}^2\leq e^{4|\lambda t|}\|\varphi\|_{\tilde{H}^s(\T^d)}^2
  \quad\text{for all}~t\in\R.
\end{align}
\indent
In view of the construction presented in Section~\ref{sec:3}, it follows from \eqref{eq:4.2}, \eqref{eq:4.3}, and limiting procedures that $\Phi(\varphi)\in (C_w\cap L^\infty_{\rm loc})(\R, H^s(\Omega))$ when $\varphi\in H^s(\Omega)$, where $\Omega\in \{\R^d,\R^d_+,\T^d\}$. Applying the argument of \cite[Remarks (c)]{KL84}, one can improve the regularity in time as $\Phi(\varphi)\in C(\R, H^s(\Omega))$ (see the proof of \cite[Lemma 2.11]{HO} for more details). This completes the proof of the last part of Theorem~\ref{thm:1.1}.

%

\begin{remark}[More general domains]
  If $\Omega$ is a smooth domain with bounded boundary, \eqref{eq:1.3}
  cannot be replaced with \eqref{eq:4.1}, and adapting the
  previous  computation  is a delicate issue. As a matter of fact, even in
  the linear case $\lambda=0$, the conservation of the $\dot H^s$ norm for
  $0<s<1$ is unclear. Typically, one may commute the fractional
  Laplacian $(-\Delta)^{s/2}$ with the (linear) Schr\"odinger
  equation, and be tempted to invoke the conservation of the $L^2$
  norm. However, the boundary condition verified by $(-\Delta)^{s/2}u$
  is unclear, and the integration by parts needed to prove the
  conservation of the $L^2$ norm of $(-\Delta)^{s/2}u$ is not obvious. 
\end{remark}

\appendix
\section{Weak solutions in the sense of Brezis}
\label{sec:A}

In this section, we state the content of \cite[Theorem 1.2 b)]{CaHa80} in a mostly self-contained manner. The construction of their $L^2$ weak solutions depend on the maximal monotone theory. We define the nonlinear operator $A$ by
\begin{align*}
Au=-i\Delta u-i\lambda u\ln(|u|^2)
\end{align*}
with the domain
\begin{align*}
D(A)=\left\{ u\in H^1_{\rm loc}(\R^d)\cap L^2(\R^d) : 
\Delta u +\lambda u\ln(|u|^2)\in L^2(\R^d) 
\r\}.
\end{align*}
From \cite[Theorem 1.1]{CaHa80} we know that $A+2|\lambda| I$ is maximal monotone in $L^2(\R^d)$. Note that the inequality \eqref{eq:1.2} is used to show the monotonicity of $A$. We now rewrite the equation \eqref{eq:1.1} as 
\begin{align}
\label{eq:A.1}
\frac{du}{dt} +(A+2|\lambda|)u=2|\lambda| u,
\quad u(0)=\varphi.
\end{align}
The definition of weak solutions in the sense of Brezis is given as follows. 
\begin{definition}[{\cite[Definition 3.1]{Brezis}}]
\label{def:A.1}
Let $A$ be a maximal monotone operator on the Hilbert space $H$. Assume $f\in L^1((0,T),H)$ for some $T>0$. We say that $u\in C([0,T],H) $ is a weak solution of the equation $\frac{du}{dt}+Au=f$ if there exists sequences $f_n\in L^1((0,T),H)$ and $u_n\in C([0,T],H)$ such that $u_n$ satisfies $\frac{du_n}{dt}+Au_n=f_n$ for a.e. $t\in (0,T)$, $f_n\to f~\text{in}~ L^1((0,T),H)$, and $u_n\to u~\text{in}~C([0,T],H)$.
\end{definition}
The authors in \cite{CaHa80} apply \cite[Theorem 3.17]{Brezis} to the equation \eqref{eq:A.1} and prove that there exists a unique $L^2$ weak solution in the sense of Definition \ref{def:A.1}. We note that in the proof of \cite[Theorem 3.17]{Brezis} the inequality (26) therein plays a key role in guaranteeing both the existence and uniqueness of solutions, and that this inequality is a consequence of monotonicity of the operator.


\section*{Acknowledgments}

The authors would like to thank Thierry Cazenave for clearly
explaining Theorem 1.2 b) in \cite{CaHa80}, Nobu Kishimoto for
pointing out the reference \cite{BO13}, and Guillaume
  Ferriere for precious comments on this paper.

R.C. is partially supported by Centre Henri Lebesgue, program
ANR-11-LABX-0020-0.  

M.H. is supported by JSPS KAKENHI Grant Number JP22K20337 and by the
Italian MIUR PRIN project 2020XB3EFL.

T.O. is supported by JSPS
KAKENHI Grant Numbers 18KK073 and 19H00644.

A CC-BY public copyright license has been applied by the authors to the present document and will be applied to all subsequent versions up to the Author Accepted Manuscript arising from this submission. 




\end{document}